\documentclass{amsart}
\let\noi=\noindent
\let\sse=\subseteq
\let\<=\langle
\let\>=\rangle
\let\limply=\Longrightarrow

	\newtheorem*{Dthm}{Theorem (Djordjevi\'{c})}
	\newtheorem*{JVthm}{Theorem (Johnson--Vinoth)}
	\newtheorem*{HKthm}{Theorem (Hartwig--Katz)}

\newsymbol\varnothing 203F

\def\0{\{0\}}

\def\span{{\kern.5pt{\rm span}\,}}

\def\CC{{\mathbb C\kern.5pt}}
\def\DD{{\mathbb D\kern.5pt}}
\def\TT{{\mathbb T\kern.5pt}}

\def\B{{\mathcal B}}
\def\C{{\kern.5pt\mathcal C}}

\def\H{{\mathcal H}}

\def\M{{\mathcal M}}
\def\N{{\mathcal N}}
\def\R{{\mathcal R}}

\def\BH{{\B[\H]}}

\newcommand{\NL}{\mathcal{NL}}

\def\smallmatrix#1{\null\,\vcenter{
                   \baselineskip=8pt\mathsurround=0pt\ialign{
                   \hfil ${\scriptstyle##}$
                   \hfil &&
                   \hfil ${\scriptstyle##}$
                   \hfil \crcr
                   \mathstrut \crcr
                   \noalign{\kern-\baselineskip}#1 \crcr
                   \mathstrut \crcr
                   \noalign{\kern-\baselineskip} \crcr }}\!}

\usepackage{amsthm}

\newtheorem{theorem}{Theorem}[section]
\newtheorem{lemma}[theorem]{Lemma}
\newtheorem{question}[theorem]{Question}
\newtheorem{corollary}[theorem]{Corollary}
\newtheorem{proposition}[theorem]{Proposition}

\theoremstyle{definition}
\newtheorem{remark}[theorem]{Remark}
\theoremstyle{example}
\newtheorem{example}[theorem]{Example}

\begin{document}

\vglue-20pt
\hfill{\it Linear Algebra and its Applications}\/
{\bf 671} (2023) 38--58
\vglue20pt
\title
[Products of Posinormal Operators]
{Closed-range posinormal operators and their products}
\author{Paul Bourdon}
\address{Department of Mathematics, University of Virginia,
         Charlottesville, USA}
\email{psb7p@virginia.edu}
\author{Carlos Kubrusly}
\address{Department of Electrical Engineering, Catholic University,
         Rio de Janeiro, Brazil}
\email{carlos@ele.puc-rio.br}
\author{Trieu Le}
\address{Department of Mathematics and Statistics, University of Toledo,
         Toledo, USA}
\email{trieu.le2@utoledo.edu}
\author{Derek Thompson}
\address{Department of Mathematics, Taylor University,
         Upland, USA}
\email{theycallmedt@gmail.com}
\subjclass{Primary 47B20; Secondary 47B15}
\thanks{{\it Keywords.}\/ Closed range, hyponormal, posinormal, 
quasiposinormal, EP, Hypo-EP operators}
\date{October 24, 2022}

\begin{abstract}
We focus on two problems relating to the question of when the product of
two posinormal operators is posinormal, giving (1) necessary conditions and
sufficient conditions for posinormal operators to have closed range, and (2)
sufficient conditions for the product of commuting closed-range posinormal
operators to be posinormal with closed range. We also discuss the
relationship between posinormal operators and EP operators (as well as
hypo-EP operators), concluding with a new proof of the Hartwig-Katz Theorem,
which characterizes when the product of posinormal operators on $\CC^n$ is
\hbox{posinormal}. 
\end{abstract}

\maketitle

\section{Introduction}

Throughout this paper the term {\it operator}\/ means a bounded linear
transformation of a Hilbert space into itself$.$ An operator is
{\it posinormal}\/ if its range is included in the range of its adjoint$.$
The class of posinormal operators includes the class of hyponormal
operators$.$ An operator is {\it quasiposinormal}\/ if the closure of its
range is included in the closure of the range of its adjoint --- if ether of
these ranges is closed, then they are both closed and the concepts of
posinormal and quasiposinormal coincide.

\vskip6pt
A necessary and sufficient condition for the product of a pair of
closed-range operators to have a closed range is given in Theorem 1 of
\cite{B}, which also provides an example of a closed-range operator whose
square does not have a closed range \cite[Corollary 5]{B}$.$ A simpler
example is given in \cite[Example 1]{BKT}$.$ The square of a posinormal
operator is not necessarily posinormal \cite[Example 1]{KVZ} but every
positive-integer power of a posinormal operator with closed range is
posinormal with closed range \cite[Corollary 14]{JV} (and so powers of a
hyponormal operator with closed range have closed range)$.$ The closed-range
assumption is crucial even if the operator and its adjoint are both
posinormal: Proposition 4.3  of \cite{BT} describes examples of
non-closed-range posinormal operators having posinormal adjoints for which
all sufficiently large powers fail to be posinormal$.$ The fact that powers
of a closed-range posinormal operator is again a closed-range posinormal
operator prompts the following question.

\begin{question}\label{Q:1.1}
{\it Is the product of two commuting posinormal operators, both with closed
range, a posinormal operator with closed range}$\,?$
\end{question}

\vskip0pt
Motivated by the preceding (still open) question, we explore necessary
conditions and sufficient conditions for posinormal operators to have closed
range, as well as investigate the structure of matrix representations of a
pair of commuting, closed range posinormal operators on a Hilbert space $\H$
relative to a natural orthogonal decomposition of $\H.$ We note that
Question \ref{Q:1.1} above has an affirmative answer if ``posinormal'' is
replaced by ``normal'' because (i) the product of two commuting normal
operators is normal, thanks to Fuglede's Theorem, and (ii) the product of
two commuting normal operators $A$ and $B$ having closed range will also
have closed range by, e.g., Proposition~\ref{P:2.2} below: the kernel of
$A$ will be reducing for $B$ by Fuglede's Theorem, showing that the operator
$Y$ of part (a) of Proposition~\ref{P:2.2} must be the zero operator, so
that the hypotheses  part (b) of Proposition~\ref{P:2.2} hold$.$ The
product of two closed-range normal operators that do not commute need not
have closed range --- see Example~\ref{TE} below.

\vskip6pt
{\it A normal operator has closed range if and only if\/ $0$ is not a limit
point of its spectrum}\/ (e.g,. set\/ ${\lambda=0}$ in\/
\cite[Proposition XI.4.5]{CFA})$.$ In Section 3, we identify three
properties of Hilbert-space operators, each one of which normal operators
possess, such that if $T$ is an operator on a Hilbert space $\H$ having
these three properties, then $T$ has closed range if and only if $0$ is not
a limit point of the spectrum of $T.$ As corollaries of this result, we
show that
\begin{itemize}
\item if $T$ is a hyponormal operator such that $0$ is not a limit point of 
the spectrum of $T$, then the range of $T$ is closed (Corollary \ref{C:3.3}),
\item if $T$ is a posinormal operator such  that $0$ is not a limit point of
the spectrum of $T$ and the restriction of $T$ to the orthogonal complement
of the kernel of $T$ is isoloid, then the range of $T$ is closed
(Corollary \ref{C:3.4}), and \item if $T$ is a posinormal operator with
closed range, then $0$ is not a limit point of the spectrum of $T$ if and
only if the adjoint of $T$ is also posinormal (Proposition \ref{P:3.5}).
\end{itemize}

\vskip6pt
If $A$ is a posinormal operator on a Hilbert space $\H$, then, by
definition, the range of $A$ is contained in the range of $A^*$, and, upon
taking orthogonal complements, we see that the kernel  $A$ is a subset of
the kernel of $A^*.$ Thus, for a posinormal operator $A$ on $\H$, the kernel
$\N(A)$ of $A$ is a subspace of $\H$ that reduces $A.$ In general, if $B$ is
an operator in the commutant of $A$ and if the kernel $\N(A)$ of $A$ reduces
$A$, then relative to the orthogonal decomposition
$\H={\N(A)^\perp\oplus\N(A)}$, the operators $A$ and $B$ have the following
matrix representations:
$$
A=\big(\smallmatrix{A' & O \cr
                    O  & O \cr}\big)
\quad\hbox{and}\quad
B=\big(\smallmatrix{B' & O \cr
                    Y  & Z \cr}\big).
$$
In Section 4 (Corollary \ref{C:4.3}), we give a sufficient condition
for the product of closed-range commuting posinormal operators to be
posinormal with closed range:
\vskip6pt\noi
{\narrower{\it
If\/ $A$ and\/ $B$ are commuting posinormal operators with closed range,
then\/ $AB$ is posinormal with closed range if\/ $B'$ and\/ $Z^*$ are
posinormal\/.
}\vskip6pt}
\noindent Moreover, using the matrix representations described above, we
show (Theorem \ref{T:4.6}) that if $A$ and\/ $B$ are commuting posinormal
operators with closed range and one of $A$ and $B$ has posinormal adjoint,
then\/ $AB$ is posinormal with closed range, a result that generalizes
\cite[Theorem 3]{DD}.

\vskip6pt
Of course, products of noncommuting posinormal operators can be posinormal$.$
 One can check that for the posinormal operators
$G=\big(\smallmatrix{1 & 1 \cr
                     0 & 1 \cr}\big)
\quad\hbox{and}\quad
P=\big(\smallmatrix{1 & 0 \cr
                    0 & 0 \cr}\big)$ on $\CC^2$, the operator $GP$ is
posinormal, but $PG$ is not (and all these operators have closed range
because all linear operators on $\CC^n$ have closed range)$.$ Another
example: on any complex Hilbert space a pair of non-commuting unitary
operators will be a pair of non-commuting, closed-range normal operators
whose product is normal with closed range (in fact, unitary). 
                    
\vskip6pt
The product of two closed-range normal operators need not have closed range: 

\begin{example} \label{TE}
There exist normal operators $A$ and $B$ having closed range such that $AB$
does not have closed range.\end{example}

\vskip0pt
Let ${(e_j)_{j=0}^\infty}$ be the natural basis of $\ell^2$ so that the
sequence $e_j$ has $1$ as its $j$-th term and zeros elsewhere$.$ Let
${\M}_{1}$ be the closure of the span of $\{e_{2k}: k = 0,1,...\}$ and
${\M}_{2}$ be the closure of the span of
$\{g_k:=e_{2k}+e_{2k+1}/(2k+1):k=0,1,\ldots\}.$ (Note that
$\{g_k:k\ge0\}$ is orthogonal.) Set $A=(I-P_{{\M}_{1}})$ (i.e., orthogonal
projection onto the orthogal complement of ${\M}_{1}$) and $B=P_{{\M}_{2}}.$
We use Bouldin's criterion, stated below, to establish that AB does not have
a closed range.   

\vskip2pt\noi
\begin{quotation}
Bouldin's Criterion \cite{B}: {\it If $S$ and $T$ are operators on
$\mathcal{H}$ having closed range then $ST$ also has closed range if and
only if the angle between $\R(T)$ and $\N(S)\cap(\N(S)\cap\R(T))^\perp$ is
positive.}
\end{quotation}

\vskip2pt\noi
Observe that $\N(A)={\M}_{1}$ and $\R(B)={\M}_{2}$, so that
$\N(A)\cap\R(B)={\M}_{1}\cap{\M}_{2}=\{0\}.$ Thus,
$(\N(A)\cap\R(B))^\perp=\H.$ We show that the angle between $\R(B)={\M}_{2}$
and $\N(A)\cap (\N(A)\cap\R(B))^\perp=\N(A)={\M}_{1}$ is $0$, showing the
range of $AB$ is not closed by Bouldin's Criterion.

\vskip6pt\noi
The angle $\theta$ between ${\M}_{1}$ and ${\M}_{2}$ of $\mathcal{H}$ is
given by
$$
\theta
=\cos^{-1}\left(\rule{0in}{.15in}\sup\{|\langle f, g\rangle|:
f\in {\M}_{1},g\in {\M}_{2},\|f\|=1=\|g\|\}\right),
$$
where $\langle\cdot,\cdot\rangle$ denotes the inner product of $\ell^2.$ For
$n\ge 0$, let
$$
f_n=e_{2n}
\quad\text{and}\quad
g_n=\frac{\left(e_{2n}+\frac{e_{2n+1}}{2n+1}\right)}{\sqrt{1+ 1/(2n+1)^2}}
$$
and observe that $(f_n)$ and $(g_n)$ are sequences of unit vectors such that
$\langle f_n,g_n\rangle=1/\sqrt{1+ 1/(2n+1)^2}\rightarrow 1$, as
$n\to\infty.$ We see the angle between ${\M}_{1}$ and ${\M}_{2}$ is $0$, as
desired. \qed

\vskip6pt 
For the normal operators $A$ and $B$ of the preceding example, it's easy to
show that $AB$ is not normal$.$ In general, it's possible to show that for
orthogonal projections $A$ and $B$, the product $AB$ is normal if and only
if $A$ and $B$ commute (and the projections $A$ and $B$ of Example~\ref{TE}
do not commute).

\vskip6pt 
The paper is organized into four more sections$.$ Notation, terminology
and auxiliary results are considered in Section 2$.$ The results
summarized above are treated in Sections 3 and 4$.$ Section 5 brings a
detailed discussion of EP operators and matrices and how they relate to
posinormal operators and matrices, concluding with a discussion of, as well
as a new proof of, the Hartwig--Katz Theorem, which characterizes when the
product of two posinormal matrices is a posinormal matrix.

\section{Notation, Terminology, and Auxiliary Results}

Let $\H$ be an infinite-dimensional complex Hilbert space$.$ If $\M$ is a
subspace of $\H$, then we let $\M^-$ denote its closure and $\M^\perp$ its
orthogonal complement$.$ The algebra of all operators on $\H$ will be
denoted by $\BH.$ For any operator ${T\kern-1pt\in\BH}$, let $\N(T)$ stand
for the kernel of $T$, which is a closed subspace of $\H$, and let $\R(T)$
stand for the range of $T.$ Let ${T^*\!\in\BH}$ denote the adjoint of
${T\in\BH}.$ Posinormal operators are \hbox{defined as follows}.

\vskip6pt
An operator ${T\kern-1pt\in\BH}$ is {\it posinormal}\/ if
$$
\R(T)\sse\R(T^*)
\qquad\hbox{(which implies $\;\N(T)\sse\N(T^*)\kern.5pt$)},
$$
and {\it quasiposinormal}\/ if
$$
\R(T)^-\!\sse\R(T^*)^-\!
\qquad\hbox{(equivalently, $\,\N(T)\sse\N(T^*)\kern.5pt$)}.
$$
Posinormal operators are quasiposinormal and the concepts coincide if
$\R(T)$ is closed$.$ If $T$ is injective, then it is quasiposinormal; if
$T^*$ is surjective (equivalently, if $T$ is injective with closed range),
then $T$ is posinormal$.$ For equivalent definitions of posinormal operators,
see, e.g., \cite[Theorem 2.1]{R1}, \cite[Theorem B]{I},
\cite[Theorem 1]{JKKP}, \cite[Proposition 1]{KD},
\cite[Definition 1]{KVZ})$.$ An operator ${T\kern-1pt\in\BH}$ is called
{\it coposinormal}\/ or {\it coquasiposinormal}\/ if its adjoint
${T^*\!\in\BH}$ is posinormal or quasiposinormal, respectively.

\vskip6pt
Posinormal operators were introduced and systematically investigated by
Rhaly in \cite{R1}, which appeared in 1994$.$ The class of posinormal
includes the hyponormal operators but is not included in the class of
normaloid operators$.$ For a comprehensive exposition on posinormal
operators see, e.g., \cite{R1} and \cite{KD}$.$ For basic properties of
posinormal operators, see, e.g., \cite[Corollary 2.3]{R1},
\cite[Propositions 3 and 4]{JKKP}, \cite[Lemma 1, Remark 2]{KD}, and
\cite[Proposition 1]{KVZ}. Those properties
required in this paper are summarized below.

\begin{proposition}\label{P:2.1}
{\it Let $T$ be a Hilbert-space operator}\/.
\begin{description}
\item{$\kern-4pt$\rm(a)}
{\it If\/ $T$ is quasiposinormal\/ $($in particular, posinormal\/$)$, then\/
$\N(T)$ reduces}\/ $T$.
\vskip4pt
\item{$\kern-4pt$\rm(b)}
{\it  The restriction of a posinormal\/ $($quasiposinormal\/$)$ operator to
a closed invariant subspace is posinormal\/ $($quasiposinormal\/$)$}\/.
\vskip4pt
\item{$\kern-4pt$\rm(c)}
{\it If\/ $T$ is quasiposinormal\/ $($in particular, posinormal\/$)$,
then}\/ ${\N(T^2)=\N(T)}$.
\end{description}
\end{proposition}
\noindent Regarding  part (a) of the preceding proposition, we note that, in
fact, a Hilbert space operator $T$ is quasiposinormal if and only if $\N(T)$
reduces $T$:
$$
\hbox{$T$ is quasiposinormal $\iff\N(T)\subseteq\N(T^*)\iff\N(T)$ reduces
$T$.}
$$

The next lemma facilitates our exploration of properties of matrix
representations of commuting posinormal operators.
  
\begin{lemma}\label{L:2.2}
Let $B$ be an operator with closed range on $\H$. Suppose that with respect
to an orthogonal decomposition $\H=\H_1\oplus \H_2$,
$$
B=\Big(\smallmatrix{B' & O \cr
                    Y  & Z \cr}\Big).
$$
\vskip-4pt\noi
Then
\begin{itemize}
\item[(a)]
$\R({B'}^{*})^{-}\subseteq \R{B'}^{*})+\R(Y^{*}|_{\N(Z^{*})}).$
As a result, if $\R(Y^{*}|_{\N(Z^{*})})\subseteq\R({B'}^{*})$, then
$\R({B'}^{*})$ is closed and hence, $\R(B')$ is closed$.$
\item[(b)]
If $B$ is also assumed to be posinormal, then
$$
\R(Z)^{-}\subseteq\R(Z^{*}).
$$
As a result, if $Z^{*}$ is quasiposinormal, then
$\R(Z^{*})^{-}\subseteq\R(Z)^{-}$ and hence $\R(Z^{*})=\R(Z)^{-}$, which
implies that $\R(Z^{*})$ is closed.
\end{itemize}
\end{lemma}

\begin{proof}
(a) Let $x\in\R({B'}^{*})^{-}$ and let $(x_n)$ be a sequence in
$\R({B'}^{*})$ that converges to $x.$ For each $n$, there exists
$u_n\in\H_1$ such that ${x_n={B'}^*u_n}$ and we have
${B^*(u_n,0)=({B'}^*u_n,0)=(x_n,0)}.$ So ${(x_n,0)\in\R(B^*)}$ for each $n$
and $(x_n,0)\to(x,0).$ Since $\R(B^*)$ is closed (because $\R(B)$ is closed)
it follows that ${(x,0)\in\R(B^*)}.$ Thus there exists
$(u,v)\in\H_1\oplus\H_2$ such that
$$
(x,0)=B^*(u,v)=({B'}^*u+Y^*v, Z^*v).
$$
This implies $Z^{*}v=0$, which shows that $Y^{*}v\in\R(Y^{*}|_{\N(Z^{*})}).$
Since $x={B'}^{*}u+Y^{*}v$, we conclude that
$$
x\in\R({B'}^{*})+\R(Y^{*}|_{\N(Z^{*})}).
$$
(b) Now assume that $B$ is posinormal$.$ Let $y\in\R(Z)^{-}.$ Then
$(0,y)\in\R(B)^{-}= R(B)\subseteq\R(B^{*})$ because $B$ is posinormal with
closed range$.$ Then there exists $(u,v)\in\H_1\oplus\H_2$ such that
$$
(0,y)=B^{*}(u,v)=({B'}^*u+Y^*v,Z^*v).
$$
This shows that $y\in\R(Z^{*})$.
\end{proof}

\vskip1pt
Let $A$ and $B$ be operators on a Hilbert space $\H.$ According to
Proposition \ref{P:2.1}(a), if $A$ is quasiposinormal (or posinormal), then
$\N(A)$ reduces $A$.

\begin{proposition}\label{P:2.2}
{\it Suppose that $A$ and\/ $B$ commute and that $\N(A)$ reduces $A$.
\vskip6pt\noi
{\rm(a)}$\,$
With respect to the decomposition\/
$\H=\N(A)^\perp\oplus\N(A)$,
$$
A=\big(\smallmatrix{A' & O \cr
                    O  & O \cr}\big)
\quad\hbox{\it and}\quad
B=\big(\smallmatrix{B' & O \cr
                    Y  & Z \cr}\big)
\quad\hbox{\it with}\quad
Y\!A'=O,
$$
and\/ ${Y=O}$ if and only if\/ $\N(A)$ reduces\/ $B$.
\vskip6pt\noi
{\rm(b)}
If\/ $\R(A)$ and $\R(B)$ are closed and
$\R(Y^*|_{\N(Z^*)})\!\sse\!\R({B'}^*)$, then
\hbox{$\kern-.5pt\R(AB)\kern-1pt$ is closed}}\/.
\end{proposition}

\begin{proof}
(a)
Consider the decomposition ${\H\kern-1pt=\kern-1pt\N(A)^\perp\!\oplus\N(A)}.$
Since $\N(A)$ \hbox{reduces $A$,}
$$
A=\big(\smallmatrix{A' & O \cr
                    O  & O \cr}\big)=A'\oplus\,O,
\quad
B=\big(\smallmatrix{B' & X \cr
                    Y  & Z \cr}\big),
\quad
BA=\big(\smallmatrix{B'\!A' & O \cr
                     Y'\!A  & O \cr}\big),
\quad
AB=\big(\smallmatrix{A'B' & A'X \cr
                     O   & O    \cr}\big)
$$ 
with $A'\!=A|_{\N(A)^\perp}$ and $B'$ in $\B[\N(A)^\perp].$ Thus if $A$
and $B$ commute, then
$$
AB=BA=\big(\smallmatrix{B'\!A & O \cr
                        O     & O \cr}\big)
=\big(\smallmatrix{A'B' & O \cr
                   O    & O \cr}\big)
=B'\!A'\oplus\,O=A'B'\oplus\,O,
$$
with $A'X=O$ and $Y\!A'=O.$ Since $A'$ is injective and ${A'X=O}$ we get
${X=O}.$ So
$$
B=\big(\smallmatrix{B' & O \cr
                    Y  & Z \cr}\big)
\quad\hbox{and so}\quad
B^*=\big(\smallmatrix{{B'}^* & Y^* \cr
                      O      & Z^* \cr}\big),
$$
and hence\/ ${Y=O}$ if and only if\/ $\N(A)$ also reduces\/ $B$.

\vskip6pt\noi
(b)
Suppose $\R(B)$ is closed in $\H$ and $\R(Y^*|_{\N(Z^*)})\sse\R({B'}^*).$
By Lemma \ref{L:2.2}(a), we conclude that $\R(B')$ is closed in
$\N(A)^{\perp}$. 

\vskip6pt\noi
Next suppose $\R(A)$ is closed in $\H$$.$ Since $\R(A)=\R(A')\oplus\0$ and
$\R(A)$ is closed,
$$
\hbox{$\R(A')$ is closed in $\N(A)^\perp$}.
$$
Then $A'$ is injective with closed range, which means it is bounded below
--- there is a positive constant $c$ such that $\|A'x\|\ge c\|x\|$ for all
$x\in\N(A)^\perp$.

\vskip6pt\noi
Now let $(x_n)$ be an arbitrary convergent sequence in $\R(A'B').$ Then,
for each $n$, ${x_n=A'B'u_n}$ for some ${u_n\in \N(A)^\perp}.$ Because
$(x_n)$ is convergent $(A'B'u_n)$ is Cauchy, which implies, because $A'$ is
bounded below, that $(B'u_n)$ is Cauchy$.$ Thus $(B'u_n)$ converges and,
because $\R(B')$ is closed, there is a $u\in\R(B')$ such that
${\lim (B'u_n)=B'u}.$ We have $A'B'u=\lim(A'B'u_n)=\lim (x_n)$ and we see
that $\R(A'B')$ is closed,
as desired.
\end{proof}

\section{Closed-Range Posinormal Operators}

In this section, we identify three properties of Hilbert-space operators,
each one of which normal operators possess, such that if $T$ is an operator
on a Hilbert space $\H$ having these three properties, then $T$ has closed
range if and only if $0$ is not a limit point of the spectrum of $T$.

\vskip6pt
Let $T$ be a bounded linear operator on a complex Hilbert space, let
$\sigma(T)$ and $\rho(T)={\CC\backslash\sigma(T)}$ denote the spectrum and
the resolvent set of $T$, respectively, and consider the standard partition
of the spectrum into $\sigma_{\kern-1ptP}(T)$, the point spectrum,
$\sigma_{\kern-1ptR}(T)$, the residual spectrum, and
$\sigma_{\kern-.5ptC}(T)$, the continuous spectrum. An operator $T$ is
{\it isoloid}\/ is every isolated point of the spectrum $\sigma(T)$ is an
eigenvalue$.$  In particular, if the spectrum $\sigma(T)$ is a
singleton $\{\lambda\}$, then $T$ is isoloid if and only if
${\sigma_{\kern-1ptP}(T)=\{\lambda\}}.$ Vacuously, every operator whose
spectrum has no isolated \hbox{point is isoloid}$.$ There exist posinormal
operators that are not isoloid, e.g., an injective weighted unilateral shift
$T_+$ on $\ell_+^2$ with weight sequence $\left(\frac{1}{k}\right)$ is a
(compact, quasinilpotent) posinormal operator (cf$.$ \cite[Section 3]{KD})
whose spectrum $\sigma(T_+)=\{0\}$ coincides with the residual spectrum
$\sigma_{\kern-1ptR}(T_+)\kern.5pt$.

\vskip6pt
{\it A normal operator has closed range if and only if zero is not a limit
point of its spectrum}\/$.$ This is a particular case of
\cite[Proposition XI.4.5]{CFA}, whose proof is based on the Spectral Theorem
as well as the Open Mapping Theorem$.$ The preceding characterization of
closed range normal operators doesn't extend to posinormal operators; in
fact, it doesn't extend to hyponormal operators$.$ For example, the forward
shift operator on $\ell^2$ is hyponormal with closed range but its spectrum
is the entire closed unit disk (so that $0$ is a limit point of the
spectrum)$.$ We seek additional conditions that ensure a posinormal operator
has closed range if and if $0$ is not a limit point of its spectrum.

\vskip6pt
If $T$ is a posinormal operator (or even a quasiposinormal operator), we
have $\N(T)\subseteq\N(T^*)$, which ensures that $\N(T)$ reduces $T.$
uppose that for some $T\in \BH$, we know $\N(T)$ reduces $T.$ Then, we have 
$$
\quad \quad T=T'\oplus 0
\quad\hbox{on}\quad
\H=\N(T)^\perp\oplus\N(T),
$$
where ${T'=T|_{\N(T)^\perp}\!\in\B[\N(T)^\perp\kern-1pt]}.$ Assuming that
$T$ has the representation $T=T'\oplus 0$ above, we will show that if we
want the condition ``$T$ has closed range'' to imply $0$ is not a limit
point of the spectrum of $T$, then it's sufficient to assume that $0$ does
not belong to the residual spectrum of $T'$, i.e., $0\not\in \sigma_R(T').$
We will also show that if we want the condition ``$0$ is not a limit point
of $\sigma(T)$'' to imply $\R(T)$ is closed, then it's sufficient to assume
that $T'$ is isoloid.
 
\vskip6pt
If $T$ is a normal operator, observe that
\vskip2pt
\begin{itemize}
\item[(i)]  $\N(T)$ reduces $T$,
\vskip2pt
\item[(ii)] $0\not\in \sigma_R\left(T|_{\N(T)^\perp}\right)$,
\vskip2pt
\item[(iii)] $T|_{\N(T)^\perp}$ is isoloid.
\end{itemize}
\vskip2pt
As for property (i), not only does $\N(T)$ reduce $T$ when $T$ is normal, we
have $\N(T)=\N(T^*).$ To see that normal operators satisfy (ii), let $T$ be
normal and observe $T|_{\N(T)^\perp}$ is normal because $\N(T)^\perp$
reduces $T$; thus, $0$ is either an eigenvalue of $T|_{\N(T)^\perp}$ (so
that $0\not\in \sigma_R\left(T|_{\N(T)^\perp}\right)$) or fails to be
eigenvalue of both $T|_{\N(T)^\perp}$ and its adjoint, and $0$'s failing to
be an eigenvalue of $T|_{\N(T)^\perp}$ means $T|_{\N(T)^\perp}$ has dense
range  (so that  $0\not\in \sigma_R\left(T|_{\N(T)^\perp}\right)).$ We have
already noted that if $T$ is normal, then $T|_{\N(T)^\perp}$ is also normal
and because all normal operators are isoloid, we see $T|_{\N(T)^\perp}$ is
isoloid; i.e., (iii) holds.  That normal operators are isoloid is a
consequence of the Spectral Theorem, and we note that with the help of the
Riesz Decomposition Theorem isoloidness can be extended to hyponormal
operators \cite[Theorem 2]{S}$.$ We say a Hilbert-space operator $T$ is of
{\it  class} $\NL$ (``normal like'') provided $T$ satisfies conditions
(i)--(iii) above.
 
\vskip6pt
We have already pointed out that property (i) of $\NL$ operators is
satisfied by any quasiposinormal operator (and hence by any posinormal and
hyponormal operator)$.$ Also, hyponormal operators satisfy property (iii) of
$\NL$ operators (because the restriction of a hyponormal operator to a
reducing subspace is hyponormal and, as we noted in the preceding paragraph,
hyponormal operators are isoloid).

\begin{theorem}\label{T:3.1}
{\it An operator of class\/ $\NL$ has closed range if and only if zero is
not a limit point of its spectrum}\/.
\end{theorem}

\begin{proof}
Suppose ${T\kern-1pt\in\B[\H]}$ is a closed-range operator on $\H$ of class
$\NL.$ Because $T$ is of class $\NL$, we know (i) $\N(T)$ reduces $T$ and
(ii) $0\not\in \sigma_{\kern-1ptR}(T|_{\N(T)^\perp}).$ Because $\N(T)$
reduces $T$, we have the decomposition
$$
T=T'\oplus O
\quad\hbox{on}\quad
\H=\N(T)^\perp\oplus\N(T),
$$
where ${T'=T|_{\N(T)^\perp}\!\in\B[\N(T)^\perp\kern-1pt]}$, so that
$\N(T')=\0$ and $\R(T')$ is closed because $\R(T)$ is closed$.$ Also observe
that the representation $T=T'\oplus O$ on $\H={\N(T)^\perp\oplus\N(T)}$ shows
that for $\lambda\ne0$, the operator $T-\lambda I$ is invertible on $\H$ if
and only if $T'-\lambda I'$ is invertible on $\N(T)^\perp$, where $I$ is the
identity on $\H$ and $I'$ is the identity on $\N(T)^\perp$.

\vskip6pt\noi
Because $0$ is not an eigenvalue of $T'$ and
$0\not\in\sigma_{\kern-1ptR}(T|_{\N(T)^\perp})$, the range of $T'$ must be
dense in $\N(T)^\perp.$ But the range of $T'$ is closed, so that $T'$ is
surjective$.$ Hence, $T'$ is invertible; that is, $0\in \rho(T').$ Because
$\rho(T')$ is open, there is an $\epsilon>0$ such that $T'-\lambda I'$ is
invertible on $\N(T)^\perp$ whenever $|\lambda| < \epsilon.$  Thus,
$T-\lambda I$ is invertible whenever $0< |\lambda|<\epsilon$ and we see $0$
is not a limit point of the spectrum of $T$.

\vskip6pt\noi
Conversely, suppose that $0$ is not a limit point of the spectrum of $T$
where $T$ is of class $\NL.$ In particular, we know that (i) $\N(T)$ reduces
$T$ and (iii) $T|_{N(T)^\perp}$ is isoloid. As we discussed in the first
paragraph of the proof, because (i) holds, we have the representation
$T=T'\oplus O$ on $\H=\N(T)^\perp\oplus\N(T)$, where $\N(T')=\{0\}$ and
for nonzero $\lambda$,  the operator $T-\lambda I$ is invertible on $\H$ if
and only if $T'-\lambda I'$ is invertible on $\N(T)^\perp.$ Because $0$ is
not a limit point of the spectrum of $T$, we see that $0$ is not a limit
point of the spectrum of $T'.$ Thus $0$ is either not in the spectrum of
$T'$ or it's an isolated point of the spectrum; however, the latter is not a
possibility --- because $T|_{N(T)^\perp}$ is isoloid, if $0$ were an
isolated spectral point, then it would be an eigenvalue but we know
$\N(T')=\{0\}.$ Thus $T'$ is invertible and it follows that
$\R(T)=\N(T)^\perp$ is closed.
\end{proof}

\vskip2pt
The class $\NL$ is constructed so that the following holds.

\begin{corollary}\label{C:3.2}
{\it If $T$ is a posinormal operator on $\H$ such that $T|_{\N(T)^\perp}$ is
a isoloid operator whose residual spectrum does not contain $0$, then $\R(T)$
is closed if and only if $0$ is not a limit point of $\sigma(T)$.}
\end{corollary}

\vskip2pt
Let $T$ be hyponormal, then for every  $\lambda \in \CC$, the operator
$T-\lambda I$ is hyponormal.  By Proposition \ref{P:2.1}(a) we know
$\N(T-\lambda I)$ reduces $T-\lambda I$ and  by \cite[Theorem 2]{S}, we know
$T-\lambda I$ is isoloid. Thus $T-\lambda I$ satisfies (i) and (iii) of
class $\NL.$ Hence, the last paragraph of the proof of Theorem \ref{T:3.1}
yields the following.

\begin{corollary}\label{C:3.3}
{\it If $T$ is a hyponormal operator on $\H$, then whenever $\lambda$ is not
a limit point of $\sigma(T)$ the range of $T-\lambda I$ is closed.}
\end{corollary}

Similarly, the hypotheses of the next corollary imply that $T$ satisfies (i)
and (iii) of class $\NL$.

\begin{corollary}\label{C:3.4}
{\it If $T$ is a posinormal  or quasiposinormal operator such  that $0$ is
not a limit point of  $\sigma(T)$ and $T|_{\N(T)^\perp}$ is isoloid, then
$\R(T)$ is closed.}
\end{corollary}

We now characterize when a posinormal operator with closed range satisfies
condition (ii) of class $\NL$. 

\begin{proposition}\label{P:3.5}
{\it Let $T$ be a posinormal operator on $\H$ having closed range$.$ The
following are equivalent:
\begin{itemize}
\item[(a)]
$0\not\in \sigma_R\left(T|_{\N(T)^\perp}\right)$;
\vskip4pt
\item[(b)]
$T|_{\N(T)^\perp}$ is invertible;
\vskip4pt
\item[(c)]
$T$ is coposinormal;
\vskip4pt
\item[(d)]
0 is not a limit point of $\sigma(T)$.
\end{itemize}}
\end{proposition}

\begin{proof} 
(a)$\implies$(b): Let $T$ be a posinormal operator on $\H$ having closed
range such that $0\not\in \sigma_R\left(T|_{\N(T)^\perp}\right).$ Because
$\N(T)$ reduces $T$, recall that we have the decomposition
$$
T=T'\oplus 0
\quad\hbox{on}\quad
\H=\N(T)^\perp\oplus\N(T),
$$
where ${T'=T|_{\N(T)^\perp}\!\in\B[\N(T)^\perp\kern-1pt]}$, so that
$\N(T')=\0$ and $\R(T')$ is closed because $\R(T)$ is closed$.$ Because $0$
is not an eigenvalue of $T'$ and we are assuming $0\not\in\sigma_{R}(T')$,
the range of $T'$ must be dense$.$ However the range of $T'$ is closed and
thus $T'$ is surjective as well as injective --- it is invertible.

\vskip6pt\noi
(b)$\implies$(c):
Because $T'$ is invertible, we have $\R(T')=\N(T)^\perp=\R(T^*)$, with the
latter equality holding because $\R(T^*)$ is closed (because $\R(T)$ is
closed)$.$ Because $\R(T')=\R(T^*)$, we see $\R(T)=\R(T^*)$, which, by
definition, yields $T$ is coposinormal.

\vskip6pt\noi
(c)$\implies$(d):
Because $T$ is coposinormal as well as posinormal, and the range of $T^*$ is
closed, we have $\R(T)=\R(T^*)=\N(T)^\perp=\R(T').$ Thus $T'$ is both
injective and surjective --- it is invertible$.$ Recall that for
$\lambda\ne 0$, $T-\lambda I$ is invertible if and only if
$T'-\lambda I'$ is invertible$.$  Because $\rho(T')$ is open, there is an
$\epsilon>0$ such that $T'-\lambda I'$ is invertible on
$\N(T)^\perp=\R(T^*)$ whenever $|\lambda|<\epsilon.$ Thus, $T-\lambda I$ is
invertible whenever $0<|\lambda|<\epsilon$ and we see $0$ is not a limit
point of the spectrum of $T$.

\vskip6pt\noi
(d)$\implies$(a):
We establish the contrapositive implication$.$ Suppose that
$0\in\sigma_R(T').$ Because we are assuming $\R(T)$ is closed, $\R(T')$ is
also closed; moreover, it's injective$.$ Thus, $T'$ is bounded below$.$
Because $0$ is in the spectrum of $T'$ (in fact in the residual spectrum),
it cannot be in the boundary of $\sigma(T')$ because that would put $0$ in
the approximate point spectrum of $T'$ (implying $T'$ is not bounded
below)$.$ Thus, $0$ is an interior point of $\sigma(T')$, so that there is
an $\epsilon>0$ such that $T'-\lambda I'$ is not invertible whenever
$|\lambda|< epsilon.$ Hence, $T-\lambda I$ is not invertible whenever
$0<|\lambda|<\epsilon.$ Thus $0$ is a limit point of $\sigma(T)$, completing
the proof.
\end{proof}

\section{Posinormal Product of Posinormal Operators}

Recall the matrix representations developed in Section 2, for a pair of
commuting operators $A$ and $B$ on a Hilbert space $\H$ for which $\N(A)$
reduces $A$:
With respect to the decomposition\/
$\H=\N(A)^\perp\oplus\N(A)$,
\begin{equation*}
\label{Eqn:ddager}
\tag{$\ddagger$}
\qquad A=\big(\smallmatrix{A' & O \cr
                           O  & O \cr}\big)
\quad\hbox{\it and}\quad
B=\big(\smallmatrix{B' & O \cr
                    Y  & Z \cr}\big)
\quad\hbox{\it with}\quad
Y\!A'=O,
\end{equation*}
and\/ ${Y=O}$ if and only if\/ $\N(A)$ reduces\/ $B$.

\vskip6pt
In this section, we use these matrix representations to obtain necessary
conditions and sufficient conditions ensuring the product of two posinormal
operators will be posinormal.

\begin{lemma}\label{L:4.1}
{\it If\/ $A$ and\/ $B$ are commuting quasiposinormal operators on\/ $\H$
having the matrix representations \eqref{Eqn:ddager} with respect to the
decomposition\/ $\H=\N(A)^\perp\oplus\N(A)$, then}
$$
({\rm a}) \quad \N(Z)\sse\N(Z^*)\cap\N(Y^*)
\qquad\hbox{\it and}\qquad
({\rm b}) \quad \N(B')\cap\N(Y)\sse\N({B'}^*).
$$
{$\kern17pt${\rm (c)}}$\kern8pt$
{\it If $B^*$ is also quasiposinormal, then the above inclusions become
identities}\/.
\end{lemma}

\proof
Suppose that a quasiposinormal $A$ commutes with $B$, then $A$ and $B$ have
matrix representations \eqref{Eqn:ddager} and
$B^*\!=\big(\smallmatrix{{B'}^* & Y^* \cr
                         O      & Z^* \cr}\big).$
Because $B$ is quasiposinormal, ${\N(B)\sse\N(B^*)}$, so that for an
arbitrary ${(u,v)\in\N(A)^\perp\oplus\N(A)}$,
$$
(B'u,Yu+Zv)=(0,0)
\quad\limply\quad
{({B'}^*u+Y^*v,Z^*v)=(0,0)}.
$$
 
\vskip0pt\noi
(a)
Set ${u=0}$ in $\N(A)^\perp.\!$ Then ${(0,Zv)}={(0,0)}$ implies
${(Y^*v,Z^*v)}={(0,0)}$ for any ${v\in\N(A)}.$ Thus
${\N(Z)\kern-1pt\sse\kern-1pt\N(Z^*)\cap\N(Y^*)}$.

\vskip6pt\noi
(b) Now set ${v=0}$ in $\N(A).$ Then ${(B'u,Yu)=(0,0)}$ implies
${({B'}^*u, 0)=(0,0)}$ for any ${u\in\N(A)^\perp}.\!$ Thus
${\N(B')\cap\N(Y)\sse\N({B'}^*)}$.

\vskip6pt\noi
(c)
If ${\N(B)=\N(B^*)}$, then the above argument shows that
$$
\N(Z)=\N(Z^*)\cap\N(Y^*)
\qquad\hbox{and}\qquad
\N(B')\cap\N(Y)=\N({B'}^*).                                    \eqno{\qed}
$$

\begin{theorem}\label{T:4.2}
{\it Let\/ $A$ and\/ $B$ be commuting quasiposinormal operators having the
matrix representations \eqref{Eqn:ddager}}\/.
\begin{description}
\item{$\kern-4pt$\rm(a)}
{\it If\/ $B'$ is quasiposinormal, then\/ $AB$ is quasiposinormal}.
\vskip4pt
\item{$\kern-4pt$\rm(b)}
{\it If\/ $Z^*$ is quasiposinormal, then\/ $AB$ has closed range whenever\/
$A$ and\/ $B$ have closed range}.
\item{$\kern-4pt$\rm(c)}
{\it If\/ $Z^*$ quasiposinormal and $A$ and\/ $B$ have closed range, then
$B'$ and $Z$ have closed range}\/.
\end{description}
{\it In other words}\/,
\vskip4pt\noi
\begin{description}
\item{$\kern-6pt$\rm(a$'$)}
{\it If\/ $B^*|_{\N(A)^\perp}\!$ is coquasiposinormal, then\/ $AB$ is
quasiposinormal}.
\vskip4pt
\item{$\kern-6pt$\rm(b$'$)}
{\it If\/ $B|_{\N(A)}$ is coquasiposinormal, then\/ $AB$ has closed range
whenever\/ $A$ and\/ $B$ have closed range}.
\item{$\kern-6pt$\rm(c$'$)}
{\it If\/ $B|_{\N(A)}$ is coquasiposinormal and $A$ and\/ $B$ have closed
range, then the above coquasiposinormal restrictions have closed range}\/.
\end{description}
\end{theorem}

\begin{proof}
Suppose that a quasiposinormal $A$ commutes with $B$, then $A$ and $B$ have
matrix representations \eqref{Eqn:ddager}:
$$
\!A=\big(\smallmatrix{A' & O \cr
                      O  & O \cr}\big)
\quad\hbox{and}\quad
B=\big(\smallmatrix{B' & O \cr
                    Y  & Z \cr}\big)
\quad\hbox{with}\quad
B^*=\big(\smallmatrix{{B'}^* & Y^* \cr
                      O      & Z^* \cr}\big),
$$
\vskip-4pt\noi
$$
\hbox{where}\quad
{B'}^*\!=B^*|_{\N(A)^\perp}
\quad\hbox{and}\quad
Z=B|_{\N(A)}.
$$
Thus (a), (b) and (c) are equivalent to (a$'$), (b$'$) and (c$'$),
respectively$.$ Also,
$$
\hbox{$A$ is quasiposinormal $\iff$ $A'$ is quasiposinormal}.
$$
\vskip-2pt

\vskip6pt\noi
(a)
Note that ${AB=BA=B'\!A'\oplus\,O=A'B'\oplus\,O.}$ Moreover, since
${\N(A')=\0}$,
$$
\N(B'\!A')=\N(B').
$$
Indeed
$x\in\N(B'A')\Leftrightarrow x\in\N(A'B')\Leftrightarrow A'B'x=0
\Leftrightarrow B'x=0\Leftrightarrow x\in\N(B').$ Now suppose
$B'$ is quasiposinormal, that is,
$$
\N(B')\sse\N({B'}^*).
$$
In this case, since $\N(B')\sse\N({B'}^*)$ and $\N(B'\!A')=\N(B')$,
$$
\N(B'\!A')=\N(B')\sse\N({B'}^*)\sse\N({A'}^*{B'}^*)=\N((B'\!A')^*),
$$
so that $B'\!A'$ is quasiposinormal$.$ Thus $A'B'$ is quasiposinormal and so
is $AB.$ Hence
$$
\hbox{$B'$ quasiposinormal $\;\limply$ $A'B'$ quasiposinormal $\iff$ $AB$
quasiposinormal}.
$$
(The assumption ``$B$ is quasiposinormal'' was not \hbox{used in item (a).)}

\vskip6pt\noi
(b)
Now suppose $B$ is also quasiposinormal$.$ Then Lemma \ref{L:4.1}(a) ensures
that ${\N(Z)\sse\N(Z^*)\cap\N(Y^*)}$, which means
$$
\hbox{$Z\;$ is quasiposinormal \quad and \quad $Y^*|_{\N(Z)}=O$}.
$$
Thus if the quasiposinormal $Z=B|_{\N(A)}$ is coquasiposinormal as well,
then $\N(Z)=\N(Z^*)$ so that $Y^*|_{\N(Z^*)}=O.$ Therefore $AB$ has closed
range whenever $A$ and $B$ have closed range according to Proposition
\ref{P:2.2}(b).

\vskip6pt\noi
(c)
Let $A$ and $B$ have closed range$.$  Because $A$ and $B$ are commuting
posinormal operators Lemma~\ref{L:4.1} yields
$\N(Z)\sse\N(Z^*)\cap\N(Y^*)\subseteq\N(Y^*).$ Thus, ${Y^*|_{\N(Z^*)}=O}$
because $\N(Z^*)\subseteq \N(Z)$ given our assumption that $Z^*$ is
quasiposinormal. Lemma \ref{L:2.2}(a) now shows that $B'$ has closed range.
Since $B$ is posinormal (i.e., quasiposinormal with closed range) and $Z^{*}$
is quasiposinormal, Lemma \ref{L:2.2}(b) shows that $Z=B|_{\N(A)}$ has
closed range.
\end{proof}

\vskip6pt
A rewriting of Theorem \ref{T:4.2} gives a partial answer to Question
\ref{Q:1.1}.

\begin{corollary}\label{C:4.3}
{\it If\/ $A$ and\/ $B$ are commuting posinormal operators with closed
range, then\/ $AB$ is posinormal with closed range if\/ $B'$ is posinormal
and\/ $Z$ is coposinormal}\/.
\end{corollary}

\begin{remark}\label{R:4.4}
 If\/ $A$ and\/ $B$ are commuting quasiposinormal operators, then
$$
\N(B')\cap\N(Y)\sse\N({B'}^*)
$$
according to Lemma \ref{L:4.1}(b)$.$ Thus Theorem \ref{T:4.2}(a) ensures
that
$$
\N(B')\sse\N(Y)
\;\;\limply\;\,
\hbox{$B'$ is quasiposinormal}
\;\;\limply\;\,
\hbox{$AB$ is quasiposinormal}.
$$
(The above holds, in particular, if ${Y\!=O}$; i.e., if $\N(A)$ also
reduces $B$)$.$ Thus (cf$.$ Proposition \ref{P:2.2}(b)$\kern.5pt$),
{\it if\/ $A$ and\ $B$ commute and\/ ${\N(B')\sse\N(Y)}$, then}
$$
\hbox{\it $A$ and $B$ posinormal with closed range}
\;\;\limply\;\,
\hbox{\it $AB$ posinormal with closed range}.
$$
\end{remark}

We can replace the commuting assumption with coincident kernels.

\begin{proposition}\label{P:4.5}
{\it If\/ $A$ and\/ $B$ are posinormal operators with closed range, and if
they have the same kernel, then their product is posinormal with closed
range.}
\end{proposition}

\proof
Let $A$ and $B$ be closed-range posinormal operators on $\H$ such that
$\N(A)=\N(B).$ Then $\N(A)$ is reducing for both $A$ and $B$ and by
Proposition \ref{P:2.2}(a), $A$ and $B$ have the following matrix
representations with respect to the decomposition
$\H=\N(A)^\perp\oplus\N(A)$:
$$
A=\big(\smallmatrix{A' & O \cr
                    O  & O \cr}\big)
\quad\hbox{ and}\quad
B=\big(\smallmatrix{B' & O \cr
                    O  & Z \cr}\big).
$$
Because $A$ and $B$ have closed range, the same is true of $A'$ and $B'.$
Moreover both $A'$ and $B'$ are injective, which means $A'$ and $B'$ are
bounded below. Thus, their products $A'B'$ and $B'A'$ are bounded below and
we see both $A'B'$ and $B'A'$ have closed range$.$ Moreover, as
$\N(A')=\N(B')=\0$, we get $\N(A'B')=\N(B'\!A')=\0$ so that $A'B'$ and
$B'\!A'$ are quasiposinormal with closed range, thus posinormal$.$ It
follows that $AB$ and $BA$ are posinormal; e.g., 
$$
\R(AB)=\R(A'B')\subseteq\R((A'B')^*)=\R({B'}^*{A'}^*)=\R(B^*A^*)=\R((AB)^*).
                                                               \eqno{\qed}
$$             

\vskip6pt
If a pair of closed-range commuting posinormal operators is such that at
least one of them is coposinormal, then their product is closed-range
posinormal.

\begin{theorem}\label{T:4.6}
{\it If a closed-range operator $A$ that is both posinormal and coposinormal
commutes with a closed-range posinormal operator $B$, then\/ $AB$ is
posinormal with closed
range}\/.
\end{theorem}

\begin{proof}
Let $A$ and $B$ be commuting, closed-range operators such that $A$ is both
posinormal and coposinormal and $B$ is posinormal$.$  Because $\N(A)$
reduces $A$ (Proposition \ref{P:2.1}(a)), $A$ and $B$ have the matrix
representations \eqref{Eqn:ddager} relative to the decomposition
$\H=\N(A)^\perp \oplus \N(A).$ Because $A$ is both posinormal and
coposinormal $\N(A) = \N(A^*).$ Therefore
$\H={\N(A)^\perp\oplus\N(A)}={\N(A^*)^\perp\oplus\N(A^*)}.$ Hence
${A'}^*=A^*|_{\N(A^*)^\perp}$ is injective (as well as $A'$)$.$ Because $A$
and $B$ commute, Proposition \ref{P:2.2}(a) ensures that ${Y\!A'=O}$ so that
${{A'}^*Y^*=O}$; equivalently, $A^*Y^*=0.$ Hence, because $A^*$ is
injective, $Y^*=O$ and therefore $Y=O.$  Hence $\N(A)$ reduces $B$ by
Proposition \ref{P:2.2}(a)$.$ This implies that $B'$ is posinormal because
$B$ is posinormal$.$ By Theorem \ref{T:4.2}(a), $AB$ is quasiposinormal;
however, ${\R(AB)}$ is closed by Proposition \ref{P:2.2}(b) because ${Y=O}$,
and thus $AB$ is posinormal.
\end{proof}

\vskip6pt
The preceding theorem generalizes Theorem 3 of \cite{DD}, a result stated in
language different from ours:

\begin{Dthm}
If\/ ${A,B\in\B[\H]}$ are EP operators with closed ranges and\/ ${AB=BA}$,
then\/ $AB$ is the $EP$ operator with a closed range also.
\end{Dthm}

\vskip6pt
Djordjevic's theorem arises in a line of investigation distinct from that
started by Rhaly in 1994 when he introduced the notion of posinormality$.$
In the final section of this paper, we briefly explore the connections
between posinormal operators and EP operators, presenting a new proof of the
Hartwig-Katz Theorem \cite{HK} characterizing ``EP matrices''$.$ 
\vskip6pt
The following corollary of our Theorem \ref{T:4.6} above is equivalent to
Djordjevi\'{c}'s theorem:

\begin{corollary}\label{C:4.7}
{\it Suppose that \/ $A$ and\/ $B$ are commuting posinormal and coposinormal
operators with closed range; then\/ $AB$ is posinormal and coposinormal
with closed range}\/.
\end{corollary}

\begin{proof}
Applying Theorem \ref{T:4.6}, we see that $AB$ has closed range and is
posinormal. Applying Theorem \ref{T:4.6} with $B^*$ playing the role of $A$
and $A^*$ playing the role of $B$, we obtain $B^*A^*$ is posinormal; i.e.,
$AB$ is coposinormal.
\end{proof}

\section{Posinormal Operators, EP Operators, and the Hartwig--Katz Theorem}

If $T$ is a posinormal operator on $\CC^n\!$, then the range of $T$ is
necessarily closed and the inclusion ${\R(T)\sse\R(T^*)}$ implies that
${\R(T)=\R(T^*)}$ because $\R(T)$ and $\R(T^*)$ have the same dimension$.$
Thus, in the finite-dimensional setting, an operator is posinormal if and
only if its range equals that of its adjoint; that is, in finite dimensions,
a posinormal operator is necessary both posinormal and coposinormal$.$
Switching to matrix language, we see that an ${n\times n}$ matrix (with
complex entries) is posinormal provided its range (column space) is the same
as the range of its conjugate transpose$.$ Such matrices are known as ``EP
matrices.''

\vskip6pt
The modern definition of ``EP matrix'' evolved from Schwerdtfeger's notion
of ``$EP_r$ matrix'' (\cite[p$.$ 130]{Sch}, 1950)$.$ For an ${n\times n}$
matrix $A$, let $A_{(j)}$ denote the $j$-th column of $A$ and $A^{(j)}$
denote the corresponding $j$-th row (${1\le j\le n}$)$.$ Schwerdtfeger
defines an ${n\times n}$ matrix $A$ of rank $r$ to be a $P_r$ matrix
provided there exist indices ${i_1,i_2,\ldots i_r}$ with
${1\le i_1<i_2<\ldots<i_r\le r}$, such that
${\{A_{(i_1)},A_{(i_2)},\ldots A_{(i_r)}\}}$ and
${\{A^{(i_1)}, A^(i_2)\},\ldots A^{(i_r)}\}}$ are both linearly independent
sets of vectors$.$ When these sets are appropriately chosen
(see p$.$ 130 of \cite{Sch}), there is reason to call the vectors these sets
contain {\it principal vectors}\/ (because they correspond to an
${r\times r}$ rank $r$ principal submatrix of $A$)$.$ Schwerdtfeger points
out that every symmetric and every skew symmetric matrix $A$ is a $P_r$
matrix$.$ Schwerdtfeger defines ``$EP_r$ matrix'' in the two sentences
following Theorem 18.1 on page 130 of \cite{Sch}:

\vskip6pt
\begin{quotation}
{\small
The notion of $P_r$ matrix may be further restricted so that it still covers
the symmetric and skew symmetric matrices as well as other types of matrices
to be mentioned later on$.$ An $n$-matrix $A$ [an $n$-matrix is an
${n\times n}$ matrix] may be called and $EP_r$ matrix if it is a $P_r$
matrix and the linear relations among its rows are the same as those among
its corresponding columns.
}
\end{quotation}

\vskip6pt\noi
Schwerdtfeger then explains what ``same linear relations'' means$.$ In
modern notation, it means that ${\N(A)=\N(A^T)}.$ In fact, Schwerdtfeger's
definition of $EP_r$ matrix may be expressed: An ${n\times n}$ matrix $A$ is
an $EP_r$ matrix provided that $\N(A)=\N(A^T).$ We see immediately, that
Schwerdtfeger has ``further restricted so that it still covers the symmetric
and skew symmetric matrices$.$'' \hbox{Note well that, e.g.,}
$$
A=\big(\smallmatrix{1 & i  \cr
                    i & -1 \cr}\big),
$$
being symmetric, is an $EP_r$ matrix$.$ However, $A$ is not an EP matrix
because its range, the one-dimensional subspace of $\CC^2$ spanned by
${\big(\smallmatrix{1 \cr
                    i \cr}\big)}$,
is {\it not}\/ the same as the range of $A^*$, which is the one-dimensional
subspace of $\CC^2$ spanned by
${\big(\smallmatrix{1  \cr
                    -i \cr}\big)}.$
Of course, if an $EP_r$ matrix $A$ has real entries, then
${\N(A)=\N(A^T)=\N(A^*)}$ and, upon taking orthogonal complements, we have
${\R(A)=\R(A^*)}.$ Thus, an $EP_r$ matrix with real entries is an EP
matrix$.$

\vskip6pt
Pearl (\cite[p$.$ 674]{Prl}, 1966) mischaracterized Schwerdtfeger's
definition of $EP_r$ matrix, stating that

\vskip6pt
\begin{quotation}
{\small
Schwerdtfeger (\cite[p$.$ 130]{Sch}) has called a square matrix of rank $r$
an $EP_r$ matrix if it satisfies the condition:
$$
AX=0\;\text{if and only if}\;A^*X =0\quad[\text{i.e.,}\;\N(A)=\N(A^*)].
$$
}
\end{quotation}

\vskip0pt\noi
It seems that since Pearl's paper appeared, most have assumed that
Schwerdtfeger's $EP_r$ matrices are precisely today's EP matrices, but
that's not quite true.

\vskip6pt
Generalizing the notion of ``EP matrix,'' Campbell and Meyer
(\cite{CampMeyer}, 1975) introduced EP operators, which may be characterized
as follows: a Hilbert space operator $T$ is EP provided that $T$ has closed
range and ${\R(T)=\R(T^*)}.$ Thus, according to the Campbell--Meyer
definition, an EP operator in $\BH$ is a closed-range operator that is both
posinormal and coposinormal$.$ It's not clear why Schwerdtfeger chose the
$E$ in his ``$EP_r$'' designation$.$ Fortunately, there is a useful
interpretation of ``EP'': an EP-operator is naturally associated with
``Equal Projectors.''

\vskip6pt
We now discuss why EP operators may be viewed as ``equal-projector''
operators$.$ Let ${T\in\BH}$ have closed range so that $\H$ has orthogonal
decompositions $\H={\R(T^*)\oplus\N(T)}$ and $\H={\R(T)\oplus\N(T^*)}.$
Thus, in particular we have
$$
\R(T)=T(\H)=T\left(\R(T^*)\oplus\N(T)\right)=T(\R(T^*)),
$$
and we see that the restriction ${T|_{\R(T^*)}}$ is an invertible operator
mapping $\R(T^*)$ onto $\R(T).$ The {\it generalized inverse}\/ $T^\dagger$
of $T$, which is also called its pseudoinverse or its Moore-Penrose inverse,
is the operator that takes an element of $\R(T)$ to its unique $T$ preimage
in $\R(T^*)$ and takes elements of $\N(T^*)$ to zero; thus,
$$
T^\dagger=(T|_{\R(T^*)})^{-1}P_{\R(T)}.
$$
Observe that 
$$
T^\dagger T=P_{\R(T^*)}
\quad\text{and}\quad
TT^\dagger=P_{\R(T)}.
$$
Hence, if $T$ is an operator with closed range, then $T$ is an EP operator
($\R(T)= \R(T^*)$) if and only if the two projectors $T^\dagger T$ and
$TT^\dagger$ are equal; that is, $T$ is an EP-operator if and only if $T$
commutes with its generalized inverse$.$ Pearl (\cite{Prl}, 1965), working
in the matrix setting, was the first to provide this characterization; in
addition, Pearl provides an explicit formula for $T^\dagger$ when $T$ is a
matrix---see Lemma 1 and Corollary 1 of \cite{Prl}.

\vskip6pt
We have seen that a closed-range operator ${T\in\B(\H)}$ is EP if and only
if $T^\dagger T-{TT^\dagger=0}.$ Itoh (\cite{Itoh}, 2005) introduced hypo-EP
operators, defining them as \hbox{follows}: ${T\in\BH}$ is hypo-EP provided
that $T$ has closed range and ${T^\dagger T-TT^\dagger\ge0}.$ Itoh
\cite[Proposition 2.1]{Itoh} then provides several conditions equivalent to
an operator's being hypo-EP, including ${T\in\B(\H)}$ is hypo-EP if and only
if $T$ has closed range and ${\R(T)\kern-1pt\sse\kern-1pt\R(T^*\kern-1pt)}.$
$\kern-1pt$Thus a hypo-EP operator is a posinormal operator with
\hbox{closed range}.

\vskip6pt
In \cite{JV}, Johnson and Vinoth provide the following sufficient condition
for a product of operators to be EP:

\begin{JVthm}
Let\/ $A$ be a hypo-EP operator and\/ ${B\in\BH}$ have closed range$.$ If\/
${\R(B)\sse\R(A)}$ and\/ ${\N(B)\sse\N(A)}$, then\/ $AB$ is hypo-EP. 
\end{JVthm}

\vskip0pt
As an immediate corollary \cite[Corollary 14]{JV}, Johnson and Vinoth
obtain the following result, which was described in the introduction of
this paper in the following equivalent way ``Every positive-integer power
of a posinormal operator with closed range is posinormal with closed range'': 

\vskip6pt\noi
{\bf Corollary (Johnson--Vinoth).}
{\it Let\/ $A$ be a hypo-EP operator on\/ $\H.$ Then\/ $A^n$ is hypo-EP for
any integer}\/ ${n\ge1}$.

\vskip6pt\noi
Djordjevi\'{c}'s paper \cite{DD}, whose Theorem 3 is Corollary \ref{C:4.7}
of the preceding section, contains a generalization of the well-known
Hartwig--Katz Theorem, which characterizes when the product of two EP
matrices is EP$.$ Here's Djordjevi\'{c}'s generalization
(\cite[Theorem 1]{DD}, 2000): 

\begin{theorem}
[{Djordjevi\'{c}'s generalization of the Hartwig--Katz Theorem}]\label{T:5.1}
{\it $\;\;$Let\/ 
${A,B\in\B[\H]}$ be EP operators$.$ Then the following statements
are equivalent}\/.
\begin{itemize}
\item[(a)]
$AB$ {\it is an EP operator}\/;
\vskip4pt
\item[(b)]
${\R(AB)=\R(A)\cap\R(B)}$ {\it and}\/ ${\N(AB)=\N(A)+\N(B)}$;
\vskip4pt
\item[(c)]
${\R(AB)=\R(A)\cap\R(B)}$ {\it and\/ $\N(AB)$ is the closure of}\/
${\N(A)+\N(B)}$.
\end{itemize}
\end{theorem}

\vskip6pt
Keep in mind that we are assuming EP (and hypo-EP) operators have
closed range.

\vskip6pt
We now turn our attention to the Hartwig--Katz Theorem. Our goal is to
provide a new, elementary proof of it based on a characterization of
${n\times n}$ matrices $A$ such that for every $n\times n$ matrix satisfying
${\R(AB)\sse\R(B)}$,
$$
\R(AB)=\R(A)\cap\R(B).
$$
Solving a problem that had been open for 25 years (see
\cite[p$.$ 98]{BasKatz}, 1969), Hartwig and Katz proved the following
result (\cite{HK}, 1997)  in which ``RS'' denotes {\it row space}:

\begin{HKthm}
Let\/ $A$ and\/ $B$ be\/ ${n\times n}$ EP matrices$.$ The
following are equivalent:
\begin{itemize}
\item[(a)]
${\R(AB)=\R(A)\cap R(B)}$ and\/ ${RS(AB)=RS(A)\cap RS(B)}$;
\vskip4pt
\item[(b)]
${\R(AB)\sse\R(B)}$ and\/ ${RS(AB)\sse\R(A)}$.
\vskip4pt
\item[(c)]
$AB$ is EP.
\end{itemize}
\end{HKthm}

\vskip0pt
Upon taking complex conjugates of all elements in a row space of an
${n\times n}$ matrix $A$ and then transposing, we obtain the column space
of $A^*$, the conjugate-transpose of $A.$ Thus the condition
${RS(AB)=RS(A)\cap RS(B)}$ of part (a) of the Hartwig-Katz Theorem is
equivalent to ${\R((AB)^*)=\R(A^*)\cap\R(B^*)}.$ Taking orthogonal
complements, we see ${\R((AB)^*)\!=\!\R(A^*)\cap\R(B^*)}$ is equivalent to
$\N(AB)\!=\!\N(A)+\N(B).$ Similarly, we see that ${RS(AB)\subseteq\R(A)}$ is
equivalent to $\N(A)\subseteq \N(AB).$ Thus, the Hartwig--Katz equivalent
conditions may be restated as follows for EP matrices $A$ and $B$:
\begin{itemize}
\item[(a)]
${\R(AB)=\R(A)\cap\R(B)}$ and ${\N(AB)=\N(A)+\N(B)}$;
\vskip4pt
\item[(b)]
${\R(AB)\sse\R(B)}$ and ${\N(A)\sse\N(AB)}$;
\vskip4pt
\item[(c)]
$AB$ is EP.
\end{itemize}

\vskip3pt\noi
We see that Djordjevi\'{c} in Theorem 5.1 has generalized to infinite
dimensional Hilbert space the equivalence of (a) and (c) of the
Hartwig--Katz Theorem but did not speak to the issue of generalizing the
equivalence of (b) and (c)$.$ Here's an example illustrating that the
equivalence of (b) and (c) does not generalize to infinite dimensions.

\vskip6pt\noi
\begin{proposition}\label{P:5.2}
{\it There exist EP operators\/ $A$ and\/ $B$ on\/ a Hilbert space $\H$ for
which\/ {\rm(i)} $\,\R(AB)\sse\R(B)$ and\/ {\rm(ii)} $\,{\N(A)\sse\N(AB)}$
such that\/ $AB$ is not EP}\/.
\end{proposition}

\begin{proof}
Let ${(e_j)_{j=0}^\infty}$ be the natural basis of $\ell^2$ so that the
sequence $e_j$ has $1$ as its $j$-th term and zeros elsewhere:
${(a_0, a_1,a_2,\ldots)}=\sum_{n=0}^\infty a_j e_j.$ Let $\M$ be the one
dimensional subspace of $\ell^2$ spanned by $e_0$ and let $P_{\M}$ be the
orthogonal projection of $\ell^2$ onto $\M.$ Let $F$ be the forward shift on
$\ell^2$,
$F\left(\sum_{j=0}^\infty a_je_j\right)=\sum_{j=0}^\infty a_je_{j+1}$, so
that $F^*$ is the backward shift,
$F^*\left(\sum_{j=0}^\infty a_je_j\right)=\sum_{j=1}^\infty a_je_{j-1}.$
Note that $F^*F=I$, while ${FF^*=P_{{\M}^\perp}}$.
 
\vskip6pt
Let ${\H=\ell^2\oplus\ell^2}$, let $I$ be the identity on $\ell^2$, and
define $A$ and $B$ on $\H$ by
$$
A=\big(\smallmatrix{0 & 0 \cr
                    0 & I \cr}\big)
\quad\hbox{and}\quad
B=\big(\smallmatrix{F & P_{\M} \cr
                    0 & F^* \cr}\big).
$$
Observe that $B$ is unitary (hence surjective):
$$
BB^*=\big(\smallmatrix{F & P_{\M} \cr
                       0 & F^* \cr}\big)
     \big(\smallmatrix{F^* & 0 \cr
                       P_{\M} & F \cr}\big)
    =\big(\smallmatrix{P_{{\M}^\perp}+P_{\M} & \;P_{\M}F \cr
                       F^*P_{\M}          & \;I    \cr}\big)
    =\big(\smallmatrix{I & 0 \cr
                       0 & I \cr}\big),
$$
while
$$
B^*B=\big(\smallmatrix{F^* & 0 \cr
                        P_{\M} & F \cr}\big)
     \big(\smallmatrix{F & P_{\M} \cr
                       0 & F^* \cr}\big)
    =\big(\smallmatrix{I    & \;F^*P_{\M}          \cr
                       P_{\M}F & \;P_{\M}+P_{{\M}^\perp} \cr}\big)
    =\big(\smallmatrix{I & 0 \cr
                       0 & I \cr}\big).
$$
Because $B$ is surjective, clearly (i) ${\R(AB)\subseteq\R(B).}$ Also,
$\N(A)$, which equals ${\ell^2 \oplus 0}$ is clearly contained in the kernel
of
$AB=\big(\smallmatrix{0 & 0   \cr
                      0 & F^* \cr}\big).$
The operator $B$ is EP because it's invertible, and $A$ is EP because it's
self-adjoint$.$ However,
$AB=\big(\smallmatrix{0 & 0 \cr
                      0 & F^* \cr}\big)$
is not EP, because its range is $0 \oplus \ell^2$, but the range of its
adjoint is $0 \oplus {\M}^\perp$, a proper subset of $\R(AB)$ (making $AB$ 
coposinormal, but not posinormal).
\end{proof}

\vskip0pt
Note that the condition ${\R(AB)=\R(A)\cap\R(B)}$ holds whenever $A$, $B$,
and $AB$ are all EP operators$.$ We rely on the following to yield the
Hartwig--Katz Theorem as a straightforward corollary.

\begin{theorem}\label{T:5.3}
{\it Let ${A\!:\CC^n\to\CC^n}$ be linear. The following are equivalent}\/.
\begin{itemize}
\item[(a)]
$\N(A^2)=\N(A)$;
\vskip4pt
\item[(b)]
{\it for each linear ${B\!:\CC^n\to\CC^n}$ satisfying}\/ $\R(AB)\sse\R(B)$,
$$
\R(AB)=\R(A)\cap\R(B);
$$
\item[(c)]
$\R(A^2)=\R(A)$;
\vskip4pt
\item[(d)]
$r(A^2)=r(A)$, {\it where $r$ denotes rank}\/.
\end{itemize}
\end{theorem}

\begin{proof}
(a)$\implies$(b): Suppose that ${\N(A^2)=\N(A)}$ and that
${B\!:\CC^n\to\CC^n}$ is a linear mapping whose range is invariant for $A$;
i.e., ${\R(AB)\sse\R(B)}.$ 

\vskip6pt
Observe that ${M\!:=\R(A)\cap\R(B)}$ is also invariant for $A.$ Consider
$$
A|_M\!:\R(A)\cap\R(B)\rightarrow\R(A)\cap\R(B).
$$
Suppose that ${v\in\R(A)\cap\R(B)}$ is such that ${Av=0}.$ Then, because
${v=Aw}$ for some ${w\in\CC^n}$, we see ${0=AAw=A^2w}$, so that
${w\in\N(A^2)=\N(A)}$ and ${0=Aw=v}.$ Thus, $A|_M$ is injective and hence
surjective$.$ We see
$$
\R(A)\cap\R(B)=A|_M(\R(A)\cap\R(B))= A(\R(A)\cap\R(B))\sse A(\R(B))=\R(AB).
$$
The reverse inclusion, ${\R(AB)\sse\R(A)\cap\R(B)}$ holds because
$\R(AB)$ is clearly contained in $\R(A)$ while ${\R(AB)\sse\R(B)}$ by
hypothesis$.$ Hence, (a)$\implies$(b).

\smallskip
To see that (b)$\implies$(c), apply (b) with ${B=A}$. That (c) implies
(d) is clear$.$ By the rank-nullity theorem, condition (d) implies that
$\N(A^2)$ and $\N(A)$ have the same dimension, but $\N(A)\sse\N(A^2)$, so
that (a) holds. We've shown (d) implies (a), completing the proof.
\end{proof}

\vskip0pt
Recall that for all posinormal operators $T$, we have $\N(T^2)=\N(T)$ by
Proposition \ref{P:2.1}(c) (which is Proposition 3 of \cite{JKKP})$.$ Thus,
condition (a) of Theorem \ref{T:5.3} holds whenever $A$ is EP.

\begin{corollary}[{Hartwig--Katz Theorem}]\label{C:5.4}
{\it Suppose that\/ $A$ and\/ $B$ are EP operators on\/ $\CC^n\!.$ Then\/
$AB$ is EP if and only if\/ {\rm(i)} $\,{\R(AB)\sse\R(B)}$ and}\/ {\rm(ii)}
$\,{\N(A)\sse\N(AB)}$.
\end{corollary}

\begin{proof}
Suppose that $AB$ is EP$.$ Then 
$$
\R(AB)=\R(B^*A^*)\sse\R(B^*)= \R(B),
$$
where the final equality holds because $B$ is EP$.$ Thus, (i) holds.
Similarly
$$
\R(B^*A^*)=\R(AB)\sse\R(A)=\R(A^*).
$$
Thus, ${\R(A^*)^\perp\sse\R(B^*A^*)^\perp}$; that is, ${\N(A)\sse\N(AB)}$,
so that (ii) holds.

\vskip6pt
Now suppose that (i) and (ii) hold for EP operators $A$ and $B.$ 
\vskip6pt
Because $A$ is EP, ${\N(A^2)=\N(A)}$; thus, since we assuming (i) holds,
Theorem \ref{T:5.3}, (a)$\implies$(b) yields
\begin{equation}\label{RAB}
\R(AB)\sse\R(A)\cap\R(B).
\end{equation}
Because $B^*$ is EP, ${\N(B^{*2})=\N(B^*)}$; taking orthogonal complements,
(ii) yields $\R(B^*A^*) \subseteq \R(A^*).$ Hence, by Theorem \ref{T:5.3},
(a)$\implies$(b), 
\begin{equation}\label{RASBS}
\R(B^*A^*) = \R(B^*)\cap\R(A^*).
\end{equation}
Because ${\R(A)=\R(A^*)}$ and ${\R(B)=\R(B^*)}$ for the EP operators $A$ and
$B$, (\ref{RAB}) and (\ref{RASBS}) yield
$$
\R(AB)=\R((AB)^*),
$$
so that $AB$ is EP, as desired.
\end{proof}

\vskip2pt
We note that Koliha (\cite{Kh}, 1999) obtained a simple proof of the
Hartwig--Katz Theorem quite different from ours.

\bibliographystyle{amsplain}

\begin{thebibliography}{10}

\bibitem{BasKatz}
T.S.\ Baskett and I.J.\ Katz, 
{\it Theorems on products of EP matrices}\/, 
Linear Algebra Appl. {\bf 2} (1969), 87--103.

\bibitem{B}
R.\ Bouldin,
{\it The product of operators with closed range}\/,
T\^ohoku Math. J. {\bf 25} (1973), 359--363.

\bibitem{BKT}
P.S.\ Bourdon, C.S.\ Kubrusly, D.\ Thompson,
{\it Powers of posinormal Hilbert-space operators}\/,
(2022) available at
https://arxiv.org/abs/2203.01473

\bibitem{BT}
P.S.\ Bourdon and D.\ Thompson,
{\it Posinormal composition operators on}\/ $\H^2$,
J. Math. Anal. Appl. {\bf 518} (2023), paper 126709, 20 pp.

\bibitem{CampMeyer} 
S.L.\ Campbell and C.D.\ Meyer,
{\it EP operators and generalized inverses}\/, 
Canad. Math. Bull. {\bf 18} (1975), 327--333.

\bibitem{CFA}
J.B.\ Conway,
{\it A Course in Functional Analysis}\/,
2nd ed. Springer, New York, 1990.

\bibitem{DD}
D.S.\ Djordjevi\'c,
{\it Products of EP operators on Hilbert spaces}\/,
Proc. Amer. Math. Soc. {\bf 129} (2000), 1727--1731.

\bibitem{HK}
R.E.\ Hartwig and I.J.\ Katz,
{\it On products of EP matrices}\/,
Linear Algebra Appl., {\bf 252} (1997), 339--345.

\bibitem{I}
M.\ Itoh,
{\it Characterization of posinormal operators}\/,
Nihonkai Math. J. {\bf 11} (2000), 97--101.

\bibitem{Itoh}
M.\ Itoh,
{\it On some EP operators}\/,
Nihonkai Math. J. {\bf 16(1)} (2005), 49--56.

\bibitem{JKKP}
I.H.\ Jeon, S.H.\ Kim, E.\ Ko and J.E.\ Park,
{\it On positive-normal operators}\/,
Bull. Korean Math. Soc. {\bf 39} (2002), 33--41.

\bibitem{JV}
P.\ Sam Johnson and A.\ Vinoth,
{\it Product and factorization of hypo-EP operators}\/,
Spec. \hbox{Matrices} {\bf 6} (2018), 376--382.

\bibitem{Kh}
J.J.\ Koliha,
{\it A simple proof of the product theorem for EP matrices}\/,
Linear Algebra Appl. {\bf 294} (1999), 213--215.

\bibitem{KD}
C.S.\ Kubrusly and B.P.\ Duggal,
{\it On posinormal operators}\/,
Adv. Math. Sci. Appl. {\bf 17} (2007), 131--148.

\bibitem{KVZ}
C.S.\ Kubrusly, P.C.M.\ Vieira, and J.\ Zanni,
{\it Powers of posinormal operators}\/,
Operators and Matrices {\bf 10} (2016), 15--27.

\bibitem{Prl}
M.H.\ Pearl,
{\it On generalized inverses of matrices}\/,
Proc. Camb. Phil. Soc. {\bf 62} (1966), 673--677.

\bibitem{R1}
H.C.\ Rhaly, Jr.,
{\it Posinormal operators}\/,
J. Math. Soc. Japan {\bf 46} (1994), 587--605.

\bibitem{Sch}
H.\ Schwerdtfeger,
{\it Introduction to Linear Algebra and the Theory of Matrices}\/,
P. Noordhoff, Groningen, 1950.

\bibitem{S}
J.G.\  Stampfli,
{\it Hyponormal operators}\/,
Pacific. J. Math. (1962), 1453--1458.

\end{thebibliography}

\end{document}